\documentclass{article}

\PassOptionsToPackage{numbers, compress, sort}{natbib}

\usepackage{caption} 
 
 \usepackage[preprint]{neurips_2020}

\usepackage[utf8]{inputenc} 
\usepackage[T1]{fontenc}    
\usepackage{url}            
\usepackage{booktabs}       
\usepackage{amsfonts}       
\usepackage{nicefrac}       
\usepackage{microtype}      

\usepackage{enumitem}
\usepackage{tcolorbox} 
\usepackage{xcolor}
\usepackage[colorlinks = true,
            linkcolor = BLUE,
            urlcolor  = BLUE,
            citecolor = BLUE]{hyperref}

\usepackage{multicol}
\usepackage{algorithm}
\usepackage{algorithmic}
\usepackage{amssymb} 
\usepackage{amsmath} 
\usepackage{wrapfig} 
\usepackage{graphicx}
\usepackage{multirow}
\usepackage{colortbl,booktabs} 
\usepackage{lipsum}   
\usepackage{amsthm} 
\usepackage{amsmath} 
\usepackage{arydshln} 
\usepackage{subfig}

\newtheorem{theorem}{Theorem}[section]

\newtheorem{lemma}{Lemma}[section]

\newtheorem{remark}{Remark}[section]
\newtheorem{definition}{Definition}[section]

\usepackage{boldline}

\newcommand{\eg}{\textit{e.g.}\ }
\newcommand{\ie}{\textit{i.e.}\ }
\newcommand{\nb}{\textit{n.b.}\ }

\newcommand{\itemsymbol}{{\small $\blacktriangleright$}}

\newcommand{\bbR}{{\mathbb R}}

\usepackage{tikz}
\usetikzlibrary{calc} 
\usetikzlibrary{arrows.meta}
\usetikzlibrary{fadings}
\usetikzlibrary{decorations.pathreplacing,decorations.markings}
\usepackage{ifthen}

\definecolor{BLUE}{rgb}{0.3,0.3,0.9}
\definecolor{RED}{rgb}{0.8,0.05,0.05}
\definecolor{GREEN}{rgb}{0.05,0.5,0.05}

\title{Equitable Continuous Organizations \\ with Self-Assessed Valuations}

\author{%
   \begin{tabular}{cp{0.5in}c}
      \textbf{Howard Heaton} &  & \textbf{Sam Green}\\
      \normalfont 
      Typal Research &  & 
      \normalfont  Semiotic AI \\
      \normalfont  Typal LLC   & & 
   \end{tabular} 
}

\begin{document}

\maketitle 

\begin{abstract}
    Organizations  are often unable to align the interests of  all  stakeholders with the financial success of the organization (\eg due to regulation). 
    However,  continuous organizations (COs) introduce a paradigm shift. COs offer  immediate liquidity, are permission-less  and  can align incentives.  CO shares are issued continuously in the form of tokens via a smart contract on a blockchain.     Token prices are designed to increase as more tokens are minted. When the share supply is low,  near-zero prices make it advantageous to buy  and hold     tokens until interest in the CO   increases, enabling a profitable sale.
    This attribute of COs, known as investment efficiency, is desirable.
    Yet, it can yield allocative inefficiency via the ``holdout problem,''  \ie   latecomers may find a CO more valuable than early tokenholders, but be unable to attain the same token holdings due to inflated prices.   With the aim of increasing overall equity, we introduce a  voting mechanism into COs.
    We show this   balances allocative and investment efficiency, 
    and may dissuade speculative trading behaviors, thereby decreasing investment risk. \end{abstract}

\section{Introduction} 

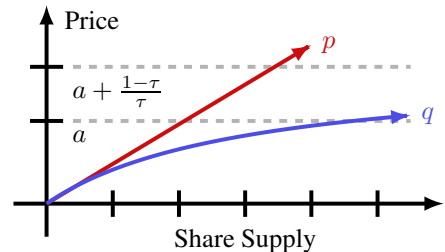
\begin{wrapfigure}[13]{r}{0.47\textwidth}
    \centering
    \vspace*{-12pt}
    \begin{tikzpicture}[scale=0.88] 
        \def\X{5.5}
        \def\A{1.25}
        \def\T{0.55}
        \def\K{0.6} 
        
        
        \draw[line width=1.5]  (-0.25, {\A + (1-\T)  / \T}) --  (0.25, {\A + (1-\T)  / \T});
        \draw[line width = 1.5, dashed, opacity=0.3] (0.4, {\A + (1-\T)  / \T}) -- (\X, {\A + (1-\T)  / \T}); 
        
        \draw[] (0.25, {\A + (1-\T)  / \T}) node[below right] {$a+\frac{1-\tau}{\tau}$} ;
        
        \draw[] (0.25, {\A }) node[below right] {$a $} ;

        \draw[line width=1.5]  (-0.25, {\A})   --  (0.25, {\A }); 
         \draw[line width = 1.5, dashed, opacity=0.3] (0.4, {\A })  -- (\X, {\A  }); 
       
        \draw[line width=1.5, -latex] (-0.5,0) -- ({\X+0.5},0);
        \draw[] ({(\X + 0.5)/2}, -0.25) node[below] {Share Supply };
        
        \draw[line width=1.5, -latex] (0,-0.5) -- (0,3.0) node[right, xshift=0.1cm, yshift=-0.2cm] {Price};
 
        \draw[line width=1.5, RED, -latex]  plot [domain=0:\X-1.5,samples=100] (\x, {\K * \x}) node[right] {${p}$};

        \draw[line width=1.5, BLUE, -latex]  plot [domain=0:\X,samples=100] (\x, { (1-\T + \T*\A)*\K*\x / (1+\T * \K *\x }) node[right] {${q}$};

        \foreach \x in {1,2,...,\X}
        { 
            \draw[line width=1.5] (\x,0.2) -- (\x, -0.2) ; 
        }         
    \end{tikzpicture}
    
    \vspace*{-8pt}
    \caption{\textcolor{RED}{Linear bonding curve $p(s)=ks$} and an   \textcolor{BLUE}{allocative curve $q(s;\ a,\tau,p)$}, with assessment $a$ and tax rate $\tau $. The curve $q$ asymptotically approaches $a + (1-\tau)   /\tau$ whereas $p \rightarrow \infty$.}
    \label{fig: eye-candy} 
\end{wrapfigure}
 
\textit{When you boil it all down, it's the investor's job to intelligently bear risk for profit. Doing it well is what separates the best from the rest.} \\[3pt]
\ \hspace*{5pt} \hfill -- Howard Marks \cite{marks2011most}
 
A new paradigm  for raising capital arose in the realm of decentralized protocols. Yet,  how can protocols (\ie {pre-programmed rules}) raise capital in a reliable and equitable manner? 
Traditional capital raising 
requires a centralized exchange to issue shares and   trading professionals to match buyers with sellers. With the advent of programmable blockchains like Ethereum, new market mechanisms are possible.  In particular, \textit{Automated Market Makers} (AMMs) encode buy-sell logic in smart contracts \cite{angeris2020improved}. Multiple AMM designs have been deployed \cite{zhang2018formal, adams2018uniswap}, beginning with Bancor \cite{wiki2022bancor}.     AMMs enable  building \textit{Continuous Organizations} (COs), where investors   directly buy and sell shares, in the form of tokens, by interacting with smart contracts 
that do \textit{not} rely on matching bid
and ask orders \cite{favre2019continuous}. This maintains liquidity irrespective of   trade volume.
Thus, AMMs enable new organizations to boostrap liquid markets. 
    But,     AMMs are susceptible to speculative behaviors (\eg  sandwich attacks\footnote{Sandwich attacks occur when an attacker executes transactions before and after a victim transaction for a guaranteed gain, at the expense of the victim.} and profit scalping\footnote{Profit scalpers buy early at low prices with the expectation of selling sometime shortly after at higher prices.}),  which may present outsized risks that   deter   long-term investors.
    We partially address these risks via a new AMM.

\begin{wraptable}[9]{r}{0.54\textwidth} 
    \centering 
    \renewcommand{\arraystretch}{1.25}
     
    \begin{tabular}{ccc}
       & \multicolumn{1}{c}{Investor} & Speculator \\\cline{2-3}  
     Payments &  10.7 $\times 10^{6}$ GRT & 4.2 $\times 10^{6}$ GRT \\
     Rewards  &  6.0 $\times 10^6$ GRT & 4.8 $\times 10^6$ GRT \\ 
    \end{tabular}
    \renewcommand{\arraystretch}{1.0}
    \vspace*{-3pt}
    \caption{On-chain statistics of The Graph indicate  speculators   realized more rewards than   payments into The Graph's COs. Speculator profit came at the expense of investors (who realized less rewards than     payments).}
    \label{tab: grt-honest-scalper} 
\end{wraptable} 
 
Events   within The Graph motivate this work. 
The Graph is a web3 protocol that provides APIs for developers to conveniently query blockchain data.   
The Graph deploys many COs\footnote{As of March 17, 2022, there are 274 COs on the decentralized network.}. 
To cause select blockchain data to become queryable through The Graph's APIs,   developers deploy one or more subgraphs. Each subgraph contains instructions 
for how to extract data from a blockchain. 
Curators are entities within The Graph and they influence which subgraphs are indexed. 
Curators influence Indexers by depositing tokens (GRT) and minting shares in   COs that are unique to each subgraph \cite{graph2022curator}.  
The Graph's COs are intended to incentivize curators to use exogenous information to determine the value of particular subgraph deployments. Despite this, the original CO design  deployed by The Graph experienced significant short-term speculative activity (\eg see example behavior in Figure \ref{fig: speculat-game-behavior}). 

\paragraph{Speculating versus Investing.}  
The intention of curation in The Graph protocol is to encourage individuals with long-term investment horizons  to provide meaningful signals for other protocol workers to follow. In this work, we call such individuals \textit{investors}. 
Our definition of investment follows Phelps' description: ``Wise investors do not buy a stock just because it is going up or is expected to go up. Wise investors buy because they foresee an increase in earning or dividends that will make today's price look cheap in years to come'' \cite{phelps2015one}. 
On the other hand, \textit{speculators} are those   seeking short-term gain via price fluctuations (\eg   exhibited by profit scalping and sandwich attacks).

The Graph launched its decentralized network in December 2020. We analyzed all 11,358 curation mint/burn transactions   from that time until September 2021. During those nine months, we measured speculator activity,  defined here  as a curator signalling within two minutes of a subgraph being published. This definition comes from assuming if a curator signals within the first two minutes of publication, then the curator could not have performed  due diligence on the subgraph\footnote{This assumes the curator does not have access to insider information.}.  
We measured so-called investor and speculator curator share purchases and   sells,  summarized by Table \ref{tab: grt-honest-scalper}.
Despite making roughly 28\% of payment volume, speculators   obtained over 44\% of total realized rewards.

\paragraph{Contribution.}  We propose a class of CO-like mechanisms   to incentivize the ``haves''  \textit{and} empower the ``have nots,''  called equitable COs (ECOs). 
Our design appears to deter sandwich attacks and profit scalping, which better rewards the ``haves''  for investing. 
Concretely, we give  three novelties.
\begin{itemize}
    \item[\itemsymbol] Propose ECOs, which augment COs by using  a voting mechanism (via allocative curves).
    
    \item[\itemsymbol] {Show   ECOs can give   scenarios that reward investors and make sandwich attacks    lose money.}    
    
    \item[\itemsymbol] Prove the tax rate parameter for allocative curves balances different types of efficiency.

\end{itemize}

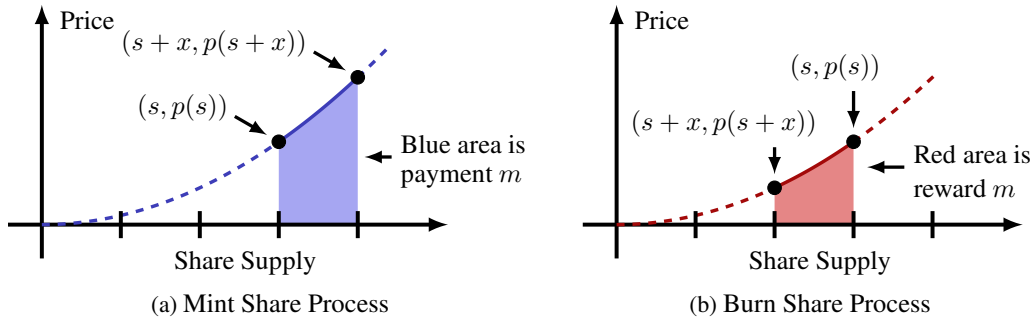
\begin{figure}[t]
    \centering
    \subfloat[{\normalsize Mint Share  Process}]{
    \begin{tikzpicture}[scale=0.9]
        \def\X{3.5}
        \def\Y{2}
        \def\O{1.16666}

        \fill [BLUE, opacity=0.5,  variable=\x, samples=100]
          (\X, 0) 
          -- plot [domain=\X:\X+\O,samples=100] (\x, { 0.1 * \x * \x})
          -- (\X+\O, 0) 
          -- cycle;  
        
        \draw[line width=1.5, -latex] (-0.5,0) -- (\X+2.5,0);%
        \draw[] ({(\X+2.5)/2},-0.25) node[below] {Share Supply};
        
        \draw[line width=1.5, -latex] (0,-0.5) -- (0,3.25) node[right, xshift=0.1cm, yshift=-0.2cm] {Price};
 
        \draw[line width=1.25, latex-] ({\X+\O+0.1}, 1.0) -- ({\X+\O+0.5},1.0);
        
        \draw[] ({\X+\O+0.5}, 1.25) node[right, yshift=-0.05cm] {Blue area is};
        
        \draw[]  ({\X+\O+0.5},0.75) node[right, yshift=-0.05cm] {payment $m$};
        
        \draw[latex-, line width=1.25] ($(\X, 0.1*\X*\X) +(150:0.2)$) -- ($(\X, 0.1*\X*\X) +(150:0.75)$) node[above left, xshift=0.05cm, yshift=-0.2cm] {$(s,p(s))$};
        
        \draw[latex-, line width=1.25] ($(\X+\O, {0.1*(\X+\O)*(\X+\O)}) +(150:0.2)$) -- ($(\X+\O, {0.1*(\X+\O)*(\X+\O)}) +(150:0.75)$) node[above left, xshift=0.05cm, yshift=-0.2cm] {$(s+x,p(s+x))$};

        \draw[BLUE!80!black, dashed, line width=1.3] plot [domain=0:\X,samples=100] (\x, { 0.1 * \x * \x});      
        
        \draw[BLUE!80!black, line width=1.3] plot [domain=\X:\X+\O,samples=100] (\x, { 0.1 * \x * \x});        
        
        \draw[BLUE!80!black, line width=1.3, dashed] plot [domain=\X+\O:\X+\O+0.5,samples=100] (\x, { 0.1 * \x * \x});

        \draw[fill=black] (\X,0.1*\X*\X) circle (0.1); 
        \draw[fill=black] ({\X+\O},{0.1*(\X+\O)*(\X+\O)}) circle (0.1); 
        
        \foreach \x in {1,2,3,4}
        { 
            \draw[line width=1.5] (1.166666*\x,0.2) -- (1.16666*\x, -0.2) ;  
        }         
    \end{tikzpicture}
    
    }
    \hspace*{10pt}
    \subfloat[{\normalsize   Burn Share Process}]{
    \begin{tikzpicture}[scale=0.9]
      
        \def\X{3.5}
        \def\Y{2}
        \def\O{1.16666}

        \fill [RED, opacity=0.5,  variable=\x, samples=100]
          (\X-\O, 0) 
          -- plot [domain=\X-\O:\X,samples=100] (\x, { 0.1 * \x * \x})
          -- (\X, 0) 
          -- cycle;  
        
        \draw[line width=1.5, -latex] (-0.5,0) -- (\X+2.5,0);%
        \draw[] ({(\X+2.5)/2},-0.25) node[below] {Share Supply};
        
        \draw[line width=1.5, -latex] (0,-0.5) -- (0,3.25) node[right, xshift=0.1cm, yshift=-0.2cm] {Price};
 
        \draw[line width=1.25, latex-] ({\X+0.2}, 0.85) -- ({\X+0.75},0.85);
        
        \draw[] ({\X+0.75},1.1) node[right, yshift=-0.05cm] {Red area is};
        
        \draw[]  ({\X+0.75},0.6) node[right, yshift=-0.05cm] {reward $m$};
        
        \draw[latex-, line width=1.25] ($(\X, 0.1*\X*\X) +(90:0.2)$) -- ($(\X, 0.1*\X*\X) +(90:0.75)$) node[above left, xshift=0.5cm, yshift=0.0cm] {$(s,p(s))$};
        
        \draw[latex-, line width=1.25] ($(\X-\O, {0.1*(\X-\O)*(\X-\O)}) +(90:0.2)$) -- ($(\X-\O, {0.1*(\X-\O)*(\X-\O)}) +(90:0.6)$) node[above left, xshift=0.7cm, yshift=0.0cm] {$(s+x,p(s+x))$};

        \draw[RED!80!black, dashed, line width=1.3] plot [domain=0:\X-\O,samples=100] (\x, { 0.1 * \x * \x});      
        
        \draw[RED!80!black, line width=1.3] plot [domain=\X-\O:\X,samples=100] (\x, { 0.1 * \x * \x});        
        
        \draw[RED!80!black, line width=1.3, dashed] plot [domain=\X:\X+1.2,samples=100] (\x, { 0.1 * \x * \x});

        \draw[fill=black] (\X,0.1*\X*\X) circle (0.1); 
        \draw[fill=black] ({\X-\O},{0.1*(\X-\O)*(\X-\O)}) circle (0.1); 
        
        \foreach \x in {1,2,3,4}
        { 
            \draw[line width=1.5] (1.166666*\x,0.2) -- (1.16666*\x, -0.2) ; 
        }         
    \end{tikzpicture}
    } 
    \caption{Canonical Bancor bonding curves. Each bonding curve $p(s)$ defines price in terms of share supply $s$. (a) Providing   payment $m$   results in minting $x$ tokens, where $x$ is chosen so the area under the curve equals $m$. (b) The reverse process, rewarding $m$, is used to burn $x$ tokens   analogously.}
    \label{fig: bancor}
\end{figure}

\newpage
\section{Continuous Organizations + Bonding Curves}

This section formalizes   bonding curves ideas in COs\footnote{For a light introduction to bonding curves, see \cite{graph2022bonding}.}.
Tokens are minted by buying and burned upon selling.
We first define bonding curves, which relate price to share supply and are denoted by $p$ and $q$.

\begin{definition}[Bonding Curve]
    A bonding curve $p$ is a strictly increasing and continuous function 
    $p\colon [0,\infty)\rightarrow\bbR$
     satisfying $p(0) = 0$.
\end{definition} 

Tokens are minted and burned via an integral equation. 
The payment/reward $m$,   tokens minted/burned $x$, and   bonding curve $p$ are related by the area under the bonding curve (see areas in Figure \ref{fig: bancor}), \ie 
\begin{equation}
    m = \int_s^{s+x}p(\zeta)
    \ \mbox{d}\zeta .
    \label{eq: buy-sell-equation}
\end{equation} 
Explicit formulas   bounding $m$ in terms of $x$ and $x$ in terms of $m$ for burning and minting, respectively, are   used to maximize computation efficiency in smart-contract code (\eg see   Appendix \ref{sec: proofs}).

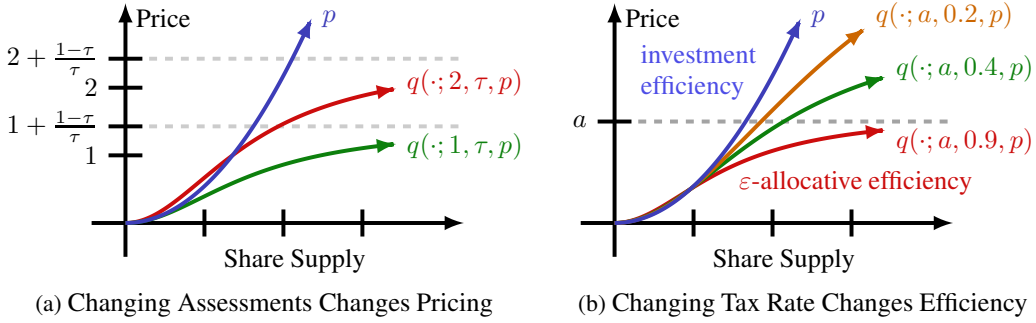
\begin{figure}[t]
    \centering
    \subfloat[{\normalsize Changing Assessments Changes Pricing}]{
    \begin{tikzpicture}[scale=0.9]
        \def\X{4}
        \def\Y{2}
        \def\O{1.16666}
        \def\T{0.7}
        \def\AA{2}
        \def\AB{1}
        \def\K{0.4}

        \foreach \yy in {\AA, \AB}
        {
         \draw[line width=1.5] (-0.25, \yy) node[left] {$\yy$} -- (0.25,\yy);
        \draw[line width=1.5] (-0.25, {\yy + (1-\T)/\T}) node[left] {$\yy+\frac{1-\tau}{\tau}$} -- (0.25,{\yy + (1-\T)/\T});
        }

        \foreach \yy in {{\AB + (1-\T)/\T}, {\AA + (1-\T)/\T}}
        {
        \draw[line width=1.5, dashed, opacity=0.2] (0.35, {\yy}) -- (\X+1,{\yy});
        }

        \draw[line width=1.5, -latex] (-0.5,0) -- (\X+1.0,0);%
        \draw[] ({(\X+1)/2},-0.25) node[below] {Share Supply};
        
        \draw[line width=1.5, -latex] (0,-0.5) -- (0,3.25) node[right, yshift=-0.2cm] {Price};

         \draw[line width=1.5, RED, -latex]  plot [domain=0:\X,samples=100] (\x, { (1-\T + \T*\AA)*\K*\x*\x / (1+\T * \K *\x * \x }) node[right, yshift=0.08cm] {${q}(\cdot; \AA,\tau,p)$};   
         
        \draw[line width=1.5, GREEN, -latex]  plot [domain=0:\X,samples=100] (\x, { (1-\T + \T*\AB)*\K*\x*\x / (1+\T * \K *\x * \x }) node[right, yshift=-0.03cm] {${q}(\cdot; \AB,\tau,p)$};    
        
        \draw[BLUE!80!black, line width=1.5, -latex] plot [domain=0:2.75,samples=100] (\x, {\K * \x * \x}) node[right] {$p$};  
        \foreach \x in {1,2,3}
        { 
            \draw[line width=1.5] (1.166666*\x,0.2) -- (1.16666*\x, -0.2);
        }         
    \end{tikzpicture}
    
    }
    \hspace*{5pt}
    \subfloat[{\normalsize   Changing Tax Rate Changes Efficiency}]{
        \begin{tikzpicture}[scale=0.9]
        
        \def\X{4.0}
        \def\Y{2}
        \def\O{1.16666}
        \def\TA{0.9}
        \def\TB{0.4}
        \def\TC{0.2}
        \def\AA{1.5} 
        \def\K{0.4}
 
        \draw[line width=1.5] (-0.25, \AA) node[left] {$a$} -- (0.25,\AA);
    
        \foreach \yy in {{\AA}, {\AA}} 
        {
        \draw[line width=1.5, dashed, opacity=0.2] (0.35, \AA) -- (\X+1,\AA);
        }

        \draw[line width=1.5, -latex] (-0.5,0) -- (\X+1.0,0);%
        \draw[] ({(\X+1)/2},-0.25) node[below] {Share Supply};
        
        \draw[line width=1.5, -latex] (0,-0.5) -- (0,3.25) node[right, yshift=-0.2cm] {Price};

         \draw[line width=1.5, RED, -latex]  plot [domain=0:\X,samples=100] (\x, { (1-\TA + \TA*\AA)*\K*\x*\x / (1+\TA * \K *\x * \x }) node[right, yshift=-0.15cm] {${q}(\cdot; a,\TA,p)$};   
         
        \draw[line width=1.5, GREEN, -latex]  plot [domain=0:\X,samples=100] (\x, { (1-\TB + \TB*\AA)*\K*\x*\x / (1+\TB * \K *\x * \x }) node[above right, yshift=-0.2cm] {${q}(\cdot; a,\TB,p)$};    
        
        \draw[line width=1.5, orange!80!black, -latex]  plot [domain=0:\X-0.3,samples=100] (\x, { (1-\TC + \TC*\AA)*\K*\x*\x / (1+\TC * \K *\x * \x }) node[above right, yshift=-0.15cm] {${q}(\cdot; a,\TC,p)$};          
        
        \draw[BLUE!80!black, line width=1.5, -latex] plot [domain=0:2.75,samples=100] (\x, {\K * \x * \x}) node[right] {$p$};  
        \foreach \x in {1,2,3}
        { 
            \draw[line width=1.5] (1.166666*\x,0.2) -- (1.16666*\x, -0.2);
        }     
        
        \draw[RED] (1.7,0.6) node[right] {$\varepsilon$-allocative efficiency};
        
        \draw[BLUE] (0.25,2.5) node[right] {investment};
        \draw[BLUE] (0.25,2.0) node[right] {efficiency}; 
    \end{tikzpicture}
    } 
    \caption{{A voting mechanism and tax rate $\tau$ enables users to augment bonding curves when buying. (a) shows a bonding curve $p$ and an allocative curve $q(s;\ a,\tau, p)$ for two different values of $a$, which are bounded above by $a+(1-\tau)/\tau$.
    Scalpers are incentivized to assess a high   $a$ (to quickly spike   price) while   investors are incentivized to assess a low   $a$ (to maximize the gap between token value and price).
    (b)  shows   increasing the tax rate $\tau$ pulls   allocative curves $q$ closer to the assessment $a$.
    Note the tax rate $\tau$ is fixed upon initialization whereas $a$ is updated during each buy transaction. 
    }}
    
    \label{fig: allocative-curves}
    \vspace*{-10pt}
\end{figure}

\paragraph{Efficiency} 
Two notions of efficiency are relevant to our discussion of COs.
Here, investment efficiency is the notion of generating as large a return as possible with a given input, relative to a time preference.
We measure   time  indirectly via the number of tokens minted. As $p$ increases more quickly wiht share supply, there is greater investment efficiency.  On the other hand, allocative efficiency occurs when everyone can buy tokens at the commonly held value  (\eg an average of assessments). 
These notions of efficiency are at odds. Investment efficiency ``increases''   as the slope of the bonding curve $p$ increases.  However, allocative efficiency requires the slope of $p$ to go to zero.  (In the next section, we balance these via a tax rate $\tau$, with low taxes giving the former and high taxes the latter.)

\begin{definition}[Investment Efficiency] 
For a  reference function $e\colon [0,\infty) \rightarrow\bbR$,
  and a share supply $S > 0$, a bonding curve $p$ is $e$-investment efficient on $[0, S]$ provided
\begin{equation}
    e(s) \leq  p(s)
    \ \ \ \mbox{for all $s\in [0,S].$}
\end{equation} 
\end{definition}

\begin{definition}[Allocative Efficiency]
    For a tolerance $\varepsilon >0$  
      an assessment $a  \geq 0$, and a share supply $S> 0$, a bonding curve $p$ is $\varepsilon$-allocatively efficient on $[0,S]$ provided
    \begin{equation}
        p(s) \leq a + \varepsilon,
        \ \ \ \mbox{for all $s \in [0,S].$}
    \end{equation}
\end{definition}

\paragraph{Self-Assessed Taxes}  
The spirit of cut and choose in the cake-cutting division game\footnote{One person cuts a cake into two pieces. The other person selects a piece; the cutter gets the remaining piece.} can be applied to   taxes. 
In the self-assessed tax methodology, the owner of an asset self-assesses its value and pays taxes based on this assessment.
The catch is the owner must be willing to sell the asset at this assessed price.  This idea has been implemented in various forms (\eg   Sun Yat-sen made such a proposal  with the state reserving the right to purchase the property at the self-assessed price \cite{niou1994analysis} and   similarly was done in New Zealand    \cite{condliffe1930new}). 
Recently, self-assessed taxes were discussed in the book \cite{posner2019radical} and \cite{weyl2016ownership},  whose ideas are credited to a proposal by Harberger \cite{harberger1965issues}. 
Self-assessed taxes can enable achieve allocative efficiency \cite{weyl2016ownership},  \ie the most efficient use of property. 
Unlike the examples above, we focus on  a fungible setting, \ie   where individual units of property are essentially interchangeable and  indistinguishable from another (\eg gold, bundles of wheat, and tokens). 
In such cases, we us an aggregate assessment $\overline{a}$ of assessments $a$ from   shareholders.

\section{Equitable Continuous Organizations} 

We aim   to improve overall equity in COs by balancing allocative and investment efficiency. To do this, we introduce \textit{Equitable Continuous Organizations} (ECOs), which are   mechanisms for minting/burning tokens that merge a voting rule and a bonding curve.
Voting rules enable individuals to self-assess  token value during minting -- the value assessment may differ from the mint price.\footnote{Recall Buffet's famous quote, attributed to Graham, ``Price is what you pay; value is what you get'' \cite{buffet2008letter}.}
Such rules can 
Countless voting rules can be used for ECOs (\eg see Algorithms \ref{alg: buy-function} and \ref{alg: sell-function}), and these will be explored by future work.
Here, we narrow   focus to   desirable features of relations between bonding curves and self-assessments. We propose relating these via the following definition.

\begin{definition}[Allocative Curve]
For a bonding curve $p$, an assessment $a \geq 0$, and a tax rate $\tau \in [0,1]$, a function $q(s, a;\ \tau,p)$ is an allocative curve for $p$ if it possesses the following properties. 
\begin{enumerate}
    \item For any share supply $S > 0$, as the tax rate $\tau$ goes to zero, the   curve $q$ uniformly approximates the bonding curve $p$ on $[0,S]$,\footnote{Although (\ref{eq: allocative-curves-prop-1}) only shows pointwise convergence, we mean  $q(\cdot, a; \ \tau,p)\rightarrow p(\cdot)$ uniformly as $\tau \rightarrow 0^+$.} \ie 
    \begin{equation}
        \lim_{\tau \rightarrow 0^+} q(s, a;\ \tau,p) = p(s),
        \ \ \ \mbox{for all $a, s \geq 0$.}
        \label{eq: allocative-curves-prop-1}
    \end{equation}

    \item As the tax rate increases to 100\%, the   curve $q$ does not exceed the assessment $a$, \ie 
    \begin{equation}
        \lim_{\tau \rightarrow 1^-} q(s, a;\ \tau,p) \leq a,
        \ \ \ \mbox{for all $a, s \geq  0$.}
        \label{eq: allocative-curves-prop-2}
    \end{equation}
    
    \item The   curve $q$ is a bounded bonding curve\footnote{Note this also means $q$ is also continuous, increasing, and $q(0,a;\ \tau,p) = 0$.}, \ie given $a > 0$, there is $B > 0$ such that
    \begin{equation}
        q(s, a;\ \tau, p) \leq B, \ \ \ \mbox{for all $s \geq 0$.}
        \label{eq: allocative-curves-prop-3}
    \end{equation}
\end{enumerate}
\end{definition}

The above properties ensure efficiency.
The first   gives investment efficiency, \ie  an allocative curve $q$ is at least as investment efficient as the bonding curve $p$ as $\tau \rightarrow 0^+.$
The second   yields allocative efficiency as the tax rate increases. Thus, $\tau$ is used to ``tax investment efficiency,'' which may improve overall equity. We formalize these claims via the lemma below (and discuss the third property after).

\begin{lemma}[Efficiency of Allocative Curves] \label{thm: efficiency-cos}
Given a tolerance $\varepsilon > 0$ and share supply $S> 0$, if $q $ is an allocative curve for a  bonding curve $p$, then there are $\tau_-,\tau_+\in (0,1)$ such that
\begin{enumerate}[label=\roman*)]
    \item 
    for all
    $\tau \in [\tau_+,1)$, 
    the   curve
    $q(s;\ a,\tau, p)$ is $\varepsilon$-allocative efficient on $[0,S]$;
        
    \item  
        for all $\tau \in (0, \tau_-]$, 
        the   curve
    $q(s;\ a,\tau, p)$ is $(p - \varepsilon)$-investment efficient on $[0,S].$
\end{enumerate}
\end{lemma}

In words, each $q$ can approximate allocative efficiency to within any tolerance $\varepsilon$, given a sufficiently high tax rate.
Conversely, as the tax rate decreases, each $q$ can approximate the original bonding curve $p$ to arbitrarily well, thereby obtaining comparable investment efficiency.
Below we provide a constructive scheme  to obtain allocative curves from bonding curves.

\begin{theorem}[Allocative Curves Exist] \label{thm: allocative-curve-bound-increasing}
For a differentiable bonding curve $p$ , a tax rate $\tau \in (0,1)$, and an assessment $a > 0$, the function $q$ defined by
\begin{equation}
    q(s;\ a,\tau,p) \triangleq  \dfrac{(1-\tau + \tau \cdot a)\cdot p(s)}{1+\tau \cdot p(s)},
    \ \ \ \mbox{for all $s \geq 0$,}
    \label{eq: allocative-curve-example}
\end{equation}
forms an allocative curve for $p$.  
\end{theorem}

A standard example is for $p$ to be linear (\ie $p(s)=ks$ for some $k> 0$). In this case, for small $s$,   $q$ in (\ref{eq: allocative-curve-example}) behaves similarly to $p$, and $q$ approaches $a + (1-\tau)/\tau$ as $s$ gets large.
In the same manner as bonding curves, we use (\ref{eq: buy-sell-equation}), replacing $p(\zeta)$ with $q(\zeta;\ a,\tau,p)$. 
For explicit formulas and   notes (\eg for Taylor approximations) for the payment/reward $m$ and tokens minted/burned $x$,  see Appendix \ref{sec: proofs}.

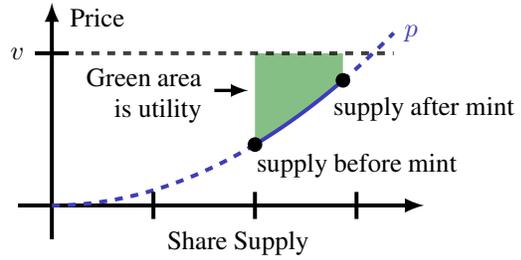
\begin{wrapfigure}[16]{r}{0.495\textwidth}

    \centering
    \begin{tikzpicture}[scale=0.9]
        \def\X{3}
        \def\Y{2}
        \def\V{2.25}
        \def\O{1.3}

        \fill [GREEN, opacity=0.5,  variable=\x, samples=100]
          (\X, \V) 
          -- plot [domain=\X:\X+\O,samples=100] (\x, { 0.1 * \x * \x})
          -- (\X+\O, \V) 
          -- cycle;  
        
        \draw[line width=1.5, -latex] (-0.5,0) -- (\X+2.5,0);%
        \draw[] ({(\X+2.5)/2},-0.25) node[below] {Share Supply};
        
        \draw[line width=1.5, -latex] (0,-0.5) -- (0,3) node[right, xshift=0.1cm, yshift=-0.2cm] {Price};
 
        \draw[line width=1.25, -latex] ({\X-0.6}, \V-0.55) -- ({\X-0.1},\V-0.55);
        
        \draw[] ({\X-0.65}, \V-0.3) node[left, yshift=-0.05cm] {Green area};
        
        \draw[]  ({\X-0.65},\V-0.8) node[left, yshift=-0.05cm] {is utility};
        
        \draw[line width=1.5] (-0.25,\V) node[left] {$v$} -- (0.25,\V);
        
        \draw[line width=1.5, opacity=0.75, dashed] (0.35,\V) -- (\X+\O+0.75,\V);

        \draw[BLUE!80!black, dashed, line width=1.5] plot [domain=0:\X,samples=100] (\x, { 0.1 * \x * \x});      
        
        \draw[BLUE!80!black, line width=1.5] plot [domain=\X:\X+\O,samples=100] (\x, { 0.1 * \x * \x});        
        
        \draw[BLUE!80!black, line width=1.5, dashed] plot [domain=\X+\O:\X+\O+0.75,samples=100] (\x, { 0.1 * \x * \x}) node[right] {$p$};

        \draw[fill=black] (\X,0.1*\X*\X) circle (0.1) node[below right, xshift=-0.1cm] {supply before mint}; 
        \draw[fill=black] ({\X+\O},{0.1*(\X+\O)*(\X+\O)}) circle (0.1)  node[below right, xshift=-0.25cm, yshift=-0.1cm] {supply after mint}; ; 
        
        \foreach \x in {1,2,3}
        { 
            \draw[line width=1.5] (1.5*\x,0.2) -- (1.5*\x, -0.2) ;  
        }         
    \end{tikzpicture}

    \vspace*{-5pt}
    \caption{For investors with   low time preferences, utility is aptly given by the difference between their valuation $v$ and  price $p$ paid (\ie area of green region). To increase this utility, assessment $a=0$ is used to ``push down price.''
    Yet, scalpers have a high time preference and  wish to pick large $a$.}
    \label{fig: investor}
\end{wrapfigure}

\begin{algorithm}[t]
\caption{ Mint Function Example for an Equitable Continuous Organization}
\label{alg: buy-function}
\begin{algorithmic}[1]           
    \STATE{\begin{tabular}{p{0.400\textwidth}r}
     \hspace*{-8pt}  Mint$(m, a;\ \overline{a}, \tau, c, \theta, r,  p):$
     & 
     $\vartriangleleft$ Input payment   $m$ and per token assessment $a$
     \end{tabular}}    \\[3pt] 
    
    \STATE{\begin{tabular}{p{0.400\textwidth}r}
      \hspace*{8pt} $\alpha  \leftarrow (1-\theta) \overline{a} + \theta a$
     & 
     $\vartriangleleft$ Update  assessment via averaging with $\theta\in(0,1)$
     \end{tabular}}     \\[3pt]    
     
      \STATE{\begin{tabular}{p{0.400\textwidth}r}
      \hspace*{8pt} $\overline{a}  \leftarrow  \max(\min(\alpha, (1+c)\overline{a}), (1-c)\overline{a})$
     & 
     $\vartriangleleft$ Bound relative assessment change by $c\in (0,1)$
     \end{tabular}}     \\[3pt]

    \STATE{\begin{tabular}{p{0.400\textwidth}r}
     \hspace*{8pt}  $s \leftarrow  s +   x$
     & 
     $\vartriangleleft$ Mint by solving (\ref{eq: buy-sell-equation}) for $x$ with $p=q(\cdot,\overline{a};\ \tau,p)$
     \end{tabular}}    \\[3pt]     
     
    \STATE{\begin{tabular}{p{0.400\textwidth}r}
     \hspace*{8pt}  $r \leftarrow r+m$
     & 
     $\vartriangleleft$ Increase reserves $r$ held in CO by $m$
     \end{tabular}}    \\[3pt]            

    \STATE{\begin{tabular}{p{0.400\textwidth}r}
     \hspace*{8pt} {\bf return} tokens $x $
     &   $\vartriangleleft$ Output tokens to ``buyer''
     \end{tabular}}   
\end{algorithmic}
\end{algorithm}

\begin{algorithm}[t]
\caption{ Burn Function Example for an Equitable Continuous Organization }
\label{alg: sell-function}
\begin{algorithmic}[1]           
    \STATE{\begin{tabular}{p{0.400\textwidth}r}
     \hspace*{-8pt}  Burn$(x;\ \overline{a},\tau , r, p):$
     & 
     $\vartriangleleft$ Input tokens $x<0$ to sell
     \end{tabular}}    \\[3pt]

    \STATE{\begin{tabular}{p{0.400\textwidth}r}
     \hspace*{4pt}  $r \leftarrow r - |m|$
     &   $\vartriangleleft$  
     Solve (\ref{eq: buy-sell-equation}) for $m<0$ with $p=q(\cdot,\overline{a};\ \tau,p)$
     \end{tabular}}     \\[3pt]       
     
    \STATE{\begin{tabular}{p{0.400\textwidth}r}
     \hspace*{4pt}  $s \leftarrow s + x$
     &   $\vartriangleleft$  Burn shares to decrease share supply\end{tabular}}   \\[3pt]       

    \STATE{\begin{tabular}{p{0.400\textwidth}r}
     \hspace*{4pt} {\bf return} reward $ |m|$
     &   $\vartriangleleft$ Output monetary reward to ``seller''
     \end{tabular}}   
\end{algorithmic}
\end{algorithm}

\paragraph{Proof-of-Vote} A voting rule property with potentially significant impact is the level of influence each mint can have on the aggregate assessment $\overline{a}$ used by the ECO. Optimal design of this remains an open question, which we highlight as follows. 
If the relative weight of an assessment  $a$ in obtaining $\overline{a}$ is proportional to the volume of tokens minted, then a few wealthy and early tokenholders may effectively control the aggregate assessment $\overline{a}$.
Alternatively, each transaction assessment $a$ could contribute equally to the aggregate $\overline{a} $  (regardless of transaction size). Informally, we refer to this as the ``proof-of-vote'' property since (relatively) fixed frictional costs\footnote{If gas prices are volatile, the influence of assessments   could be   proportional to gas prices paid.}   involved for each mint  gives a similar flavor to ``proof-of-work'' mechanisms.
In other words, the ``have nots'' can have equal contribution to defining the bonding curve as the ``haves.''
This latter option may limit the ability to make short-term gains while also stabilizing the convergence of price to an aggregate average assessment (\ie     bound the derivative of $q(s,\overline{a}; \ \tau,p)$ with respect to   $a$). 
 
 \vspace*{-5pt}

\paragraph{Limited Upsides and Front-Running} The third property of allocative curves can help align incentives. 
To maximize the value obtained from minting, investors may provide   assessments $a$ less than their valuation $v$. This is shown by Figure \ref{fig: investor}, where we note the utility for an investor that seeks to be a long-term tokenholder with low time preference is roughly their valuation $v$ minus the price paid $m$.   This can mitigate sandwich attacks since it can imply token prices will be \textit{lower} after a   purchase rather than higher (\eg see Figure \ref{fig: front-running-fail}). In such a case, a sandwich attack would  \textit{lose money}. Investors can let speculative curators  push the price up over time. Further, since the allocative curve $q$ is bounded by a function of  $\overline{a}$,   purchaser assessments $a$ may limit potential profits by low-balling token value. Depending on the minting mechanism, this may incline them to report more truthfully.

 \vspace*{-8pt}

\section{Conclusions}
We propose use of ECOs in place of COs to increase allocative efficiency and overall equity. The new allocative curves   enable desirable properties for ECOs, which may be promising both with respect to efficiency and to disincentivizing  some trading behaviors (\eg sandwich attacks) while rewarding long-term investors.
Future research will  provide explicit utility functions modeling speculators and investors to formally establish early-stage incentives, use game-theoretic analysis (\eg via Nash equilibria) to suggest various voting rules\footnote{We suspect there are more apt mint/burn schemes than those in Algorithms \ref{alg: buy-function} and \ref{alg: sell-function}.},  
and  generate realistic simulations for deeper insights.

\newpage
\begin{figure}[t]
    \centering 
        \begin{tikzpicture}[scale=0.9]
        \def\X{5.0}
        \def\Y{2}
        \def\O{1.16666}
        \def\TA{0.7}  
        \def\AA{2.5} 
        \def\AB{2.5*2.0/3.0}
        \def\K{0.4}
 
        \def\BM{{ (1-\TA + \TA*\AA)*\K*3*3 / (1+\TA * \K *3 * 3) }}
        \def\AM{{ (1-\TA + \TA*\AB)*\K*4.5*4.5 / (1+\TA * \K *4.5 *4.5) }}
 
        \draw[line width=1.5] (-0.25, \BM) node[left, yshift=0.1cm] {pre-mint price} -- (0.25,\BM); 
        \draw[line width=1.5, dashed, opacity=0.5] (0.35, \BM) -- (\X+1,\BM);
        
        \draw[line width=1.5] (-0.25, \AM) node[left, yshift=-0.1cm] {post-mint price} -- (0.25,\AM); 
        \draw[line width=1.5, dashed, opacity=0.5] (0.35, \AM) -- (\X+1,\AM);
        
        \draw[line width=1.5, dashed, opacity=0.5] (4.5, 0.35) -- (4.5,\AM);
        
        \draw[line width=1.5, dashed, opacity=0.5] (3, 0.35) -- (3,\BM);

        \draw[line width=1.5, -latex] (-0.5,0) -- (\X+1.0,0);%
        \draw[] ({(\X+1)/2},-0.75) node[below] {Share Supply};
        
        \draw[line width=1.5, -latex] (0,-0.5) -- (0,3.25) node[right, yshift=-0.2cm] {Price};

         \draw[line width=1.5, RED, -latex]  plot [domain=0:\X,samples=100] (\x, { (1-\TA + \TA*\AA)*\K*\x*\x / (1+\TA * \K *\x * \x }) node[right, yshift=0.05cm] {${q}(\cdot; a,\tau,p)$};   
         
        \draw[line width=1.5, GREEN, -latex]  plot [domain=0:\X,samples=100] (\x, { (1-\TA + \TA*\AB)*\K*\x*\x / (1+\TA * \K *\x * \x }) node[right, yshift=-0.3cm, xshift=-0.1cm] {${q}(\cdot; 2a/3,\tau,p)$};

        \foreach \x in {1,2,3}
        { 
            \draw[line width=1.5] (1.5*\x,0.2) -- (1.5*\x, -0.2);
        }    
        \draw[] (3,-0.45) node[] {$s$};
        \draw[] (4.5,-0.45) node[] {$s+x$};
        
        \draw[fill=black] (3.0, \BM) circle (0.1);
        \draw[fill=black] (4.5, \AM) circle (0.1);
        
    \end{tikzpicture}
    \caption{The voting rule can make prices \textit{lower} after minting. If $x$ tokens are purchased and the assessment provided by the new tokenholder was chosen in a manner that reduced the aggreggate assessment in the ECO from $a$ to $2a/3$, then this plot shows the new price after minting can be reduced. In such situations,  a front-runner could \textit{lose} money by  trying to mint and burn immediately before and after the $x$ tokens were minted.} 
    \label{fig: front-running-fail} 
\end{figure}
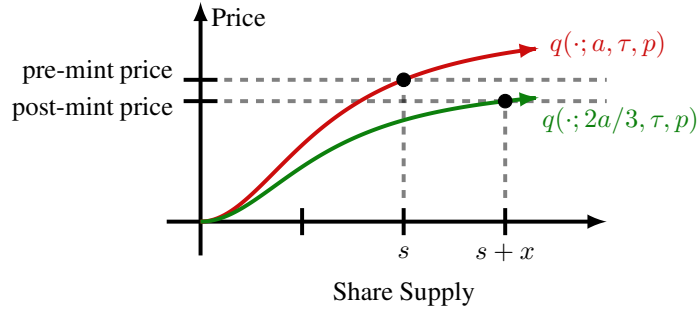

\begin{figure}[t]
    \centering
    
        \begin{tikzpicture}[scale=1.0] 
            \def\X{10.0}
     
            \draw[line width=1.5, -latex] (-0.5,0) -- (\X+1.0,0);%
            \draw[] ({(\X+1)/2},-0.75) node[below] {Time $t$: Blocks since CO Launched};
            \draw[line width=1.5, -latex] (0,-0.5) -- (0,4.0) node[right, yshift=-0.2cm] {Value of Shares Held};
            
            \draw[line width=1.5] (-0.2, 1) -- (0.2, 1);
            \draw[line width=1.5] (-0.2, 2) -- (0.2, 2);
            \draw[line width=1.5] (-0.2, 3) -- (0.2, 3);
            
            \foreach \x in {1,2,...,\X}
            {   \ifodd \x {} \else {\draw[line width=1.5] (\x,0.2) -- (\x, -0.2);} \fi
            }    
            %
            \coordinate[] (st0) at (0,1);
            \coordinate[] (st1) at (4,1);
            \coordinate[] (st2) at (4,3);
            \coordinate[] (st3) at (6,3);
            \coordinate[] (st4) at (6,0);
            
            \coordinate[] (it0) at (4,2);
            \coordinate[] (it1) at (6,2);
            \coordinate[] (it2) at (6,1);
            \coordinate[] (it3) at (10,1);
            
            \begin{scope}[very thick,decoration={
                        markings, 
                        mark=at position 0.5 with {\arrow{latex}}}
                        ]  
                        
            \foreach[evaluate=\kk as \kk using int(\k+1)] \k in {0,1,...,3}
            {   \draw[RED, line width=1.5, postaction={decorate}] (st\k) -- (st\kk);
                \ifodd \k {} \else {\fill[RED!90!black] (st\k) circle (0.12);} \fi
                \ifodd \kk {} \else {\fill[RED!90!black] (st\kk) circle (0.12);} \fi 
            }
            \foreach[evaluate=\kk as \kk using int(\k+1)] \k in {0,1,...,2}
            {   \draw[BLUE, line width=1.6, postaction={decorate}] (it\k) -- (it\kk);
                \ifodd \k {} \else {\fill[BLUE!90!black] (it\k) circle (0.12);} \fi
                \ifodd \kk {} \else {\fill[BLUE!90!black] (it\kk) circle (0.12);} \fi 
            }        
       \end{scope} 
        \draw[] (st0) node[above right, yshift=0.1cm] {{\small \textcolor{RED!80!black}{Speculator} enters at $t=0$}};
        \draw[] (st2) node[above right, yshift=0.0cm] {{\small \textcolor{BLUE!80!black}{Investor} purchase increases value of shares}};
        \draw[] (st4) node[above right, yshift=0.1cm] {{\small \textcolor{RED!80!black}{Speculator} sells shares with profit}};
        
        \draw[] (it0) node[above left, xshift=-0.1cm, yshift=0.5cm] {{\small \textcolor{BLUE!80!black}{Investor} enters after}};
        \draw[] (it0) node[above left, xshift=-0.1cm, yshift=0.0cm] {{\small due diligence}};
        \draw[] (it2) node[above right, yshift=0.1cm] {{\small \textcolor{BLUE!80!black}{Investor} now has (unrealized) loss}};
        
    \end{tikzpicture}
    \caption{The standard bonding curve design encourages speculation. In this figure, we depict profit scalping. Imagine a bot that monitors Ethereum for new bonding curves. If the bot is the first to enter a new bonding curve, it is guaranteed a profit if it sells immediately after one significant investor follows it. This leaves the investor with an unrealized loss.}
    \label{fig: speculat-game-behavior} 
\end{figure}
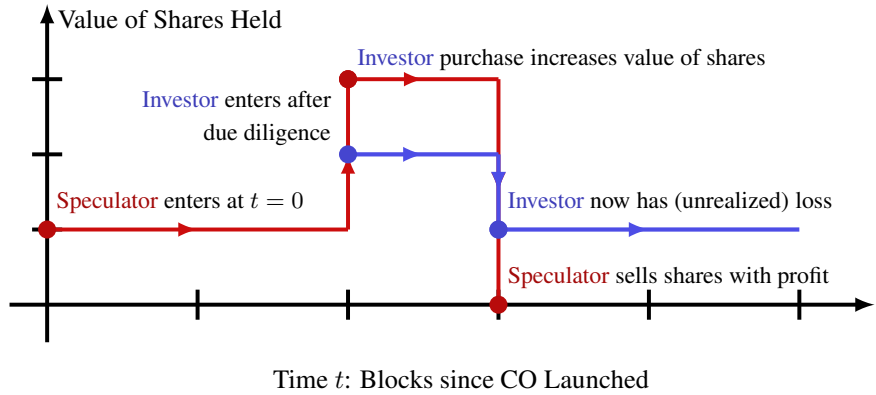

\section*{Acknowledgments}
The inspiration and many references used to develop this work came   from Brandon Ramirez at Edge \& Node, to whom we owe many thanks.
We thank Samy Wu Fung at Colorado School of Mines and Daniel McKenzie at University of California Los Angeles for valuable conversations that were constructive for this work. We thank Alexis Asseman at Semiotic AI for providing curation statistics in The Graph protocol.

\bibliographystyle{abbrv}
\bibliography{refs}
 
\newpage
\appendix

\newpage

\section{Proofs} \label{sec: proofs}

\subsection{General Allocative Curves}
{\bf Theorem \ref{thm: allocative-curve-bound-increasing}}  (Allocative Curves Exist).  \textit{For a differentiable bonding curve $p$ , a tax rate $\tau \in (0,1)$, and an assessment $a > 0$, the function $q$ defined by
\begin{equation}
    q(s;\ a,\tau,p) \triangleq  \dfrac{(1-\tau + \tau \cdot a)\cdot p(s)}{1+\tau \cdot p(s)},
    \ \ \ \mbox{for all $s \geq 0$,} 
\end{equation}
forms an allocative curve for $p$.  }

\begin{proof} We proceed in the following manner.
First we verify $q$ satisfies the third property of allocative curves (Step 1). Then we establish the second property for allocative efficiency   (Step 2). This is followed by showing $q \rightarrow p$ uniformly as $\tau \rightarrow 0^+$ (Step 3).

    {\bf Step 1.} Since $p$ is a bonding curve, $p(0) = 0$ and $p(s) > 0$ for $s > 0$. Thus, $q(0) = 0$ and $q$ is the quotient of differentiable functions, with a denominator that is bounded away from zero, and so $q$ is differentiable.  Since $q(0) = 0$ and $p(s) > p(0) = 0$ for $s > 0$ (\ie $p$ is strictly increasing),
    \begin{equation}
        q(s)
        = \dfrac{(1-\tau + \tau \cdot a)\cdot p(s)}{1+\tau \cdot p(s)}
        \leq (1-\tau + \tau \cdot a) \cdot \dfrac{1}{\tau}
        = a + \dfrac{1-\tau}{\tau},
        \ \ \ \mbox{for all $s \geq 0$.}
        \label{eq: curves-exist-proof-01}
    \end{equation}
    That is, $q $ is bounded from above by $a + (1-\tau)/\tau$.\\

    Rearranging the expression for $q$ reveals
    \begin{equation}
        q(s) = (1-\tau + \tau\cdot[a-q(s)]) \cdot p(s).
        \label{eq: curves-exist-proof-02}
    \end{equation}
    and differentiating yields (\nb we adopt the dot notation for derivatives)
    \begin{equation}
        \dot{q}(s) = (1-\tau + \tau\cdot[a-q(s)]) \cdot \dot{p}(s). - \tau \cdot p(s)\cdot \dot{q}(s),
        \ \ \ \mbox{for all $s > 0$,}
        \label{eq: curves-exist-proof-03}
    \end{equation}
    and so
    \begin{equation}
        \dot{q}(s) = \dfrac{(1-\tau + \tau\cdot[a-q(s)])\dot{p}(s)}{1+\tau \cdot p(s)},
        \ \ \ \mbox{for all $s > 0$.}
        \label{eq: curves-exist-proof-04}
    \end{equation}
    The upper bound on $q(s)$ in (\ref{eq: curves-exist-proof-01}) implies
    \begin{equation}
        1 - \tau + \tau \cdot [a-q(s)] \geq 0,
        \ \ \ \mbox{for all $s\geq 0.$}
        \label{eq: curves-exist-proof-05}
    \end{equation}
    As the denominator of $\dot{q}$ is always positive, $\dot{q}(s) > 0,$ \ie $q$ is strictly increasing. 
    Combining our results reveals $q$ is a bounded bonding curve.
    \\
    
    {\bf Step 2.}  Let $\varepsilon > 0$ be given. To verify (\ref{eq: allocative-curves-prop-2}), it suffices to show there is $\tau_+ \in (0,1)$ such that
    \begin{equation}
        q(s;\ a,\tau, p) \leq a + \varepsilon,
        \ \ \ \mbox{for all $s\geq 0$ and $\tau \in [\tau_+,1)$.}
        \label{eq: curves-exist-proof-06}
    \end{equation}
    Picking
    \begin{equation}
        \overline{\tau}_+ \triangleq \dfrac{1}{1+\varepsilon},
        \label{eq: curves-exist-proof-07}
    \end{equation}
    implies
    \begin{equation}
        1 > \tau \geq \overline{\tau}_+ = \dfrac{1}{1+\varepsilon}
        \ \ \ \implies \ \ \ 
        \tau (1+\varepsilon) \geq 1
        \ \ \ \implies \ \ \ 
        \varepsilon \geq \dfrac{1-\tau}{\tau},
        \label{eq: curves-exist-proof-08}
    \end{equation}
    and so
    \begin{equation}
        q(s;\ a,\tau, p)
        \leq a + \dfrac{1-\tau}{\tau}
        \leq a + \varepsilon,
        \ \ \ \mbox{for all $s \geq 0$ and $\tau \in [\overline{\tau}_+,1)$.}
        \label{eq: curves-exist-proof-09}
    \end{equation} 
    In particular, this shows the limit in (\ref{eq: allocative-curves-prop-2}) holds, taking $\tau_+ = \overline{\tau}_+$.  \\

    {\bf Step 3.}
    Let $\delta > 0$ be given. It suffices to show there is $\tau_- \in [0,1]$ such that
    \begin{equation}
        \|q(\cdot,a;\ \tau,p)-p(\cdot)\|_\infty
        =
        \sup_{s\in[0,S]} |q(s,a;\ \tau,p) - p(s)| \leq \delta,
        \ \ \ \mbox{for all $\tau \in [0,\tau_-].$}
        \label{eq: curves-exist-proof-10}
    \end{equation}
    Observe
    \begin{subequations}
    \begin{align}
        \left| p(s) - q(s,a;\ \tau,p)\right|
        & = \left|\dfrac{[1+\tau\cdot p(s)] \cdot p(s)}{1+\tau\cdot p(s)}
        -  \dfrac{(1-\tau + \tau \cdot a)\cdot p(s)}{1+\tau \cdot p(s)} \right|\\
        & = \tau \cdot \left| \dfrac{(1 +  p(s)-a)\cdot p(s)}{1+\tau\cdot p(s)}\right| \\
        & \leq \tau \cdot |2-a|\cdot p(s) \\
        & \leq \tau\cdot |2-a| \cdot p(S),
        \ \ \ \mbox{for all $s\in[0,S],$}
    \end{align}
    \label{eq: curves-exist-proof-11}
    \end{subequations}
    where the final inequality holds since $p$ is strictly increasing.
    Hence picking
    \begin{equation}
        \overline{\tau}_- = \dfrac{\delta}{\max(1,|2-a|\cdot p(S))}
        \label{eq: curves-exist-proof-12}
    \end{equation}
    yields
    \begin{equation}
        \left| p(s) - q(s,a;\ \tau,p)\right|
        \leq \dfrac{\delta \cdot |2-a|\cdot p(S)}{\max(1, |2-a|\cdot p(S))}
        \leq \delta,
        \ \ \ \mbox{for all $s \in [0,S]$ and $\tau \in [0,\overline{\tau}_-].$}
        \label{eq: curves-exist-proof-13}
    \end{equation}
    The bound in (\ref{eq: curves-exist-proof-13}) holds for all $s\in [0,S]$, thereby implying (\ref{eq: curves-exist-proof-10}) holds upon taking $\tau_- = \overline{\tau}_-.$
\end{proof}


\vspace*{1.in}

 {\bf Lemma \ref{thm: efficiency-cos}} (Efficiency of Allocative Curves).\textit{Given a tolerance $\varepsilon > 0$ and share supply $S> 0$, if $q $ is an allocative curve for a  bonding curve $p$, then there are $\tau_-,\tau_+\in (0,1)$ such that
\begin{enumerate}[label=\roman*)]
    \item 
    for all
    $\tau \in [\tau_+,1)$, 
    the   curve
    $q(s;\ a,\tau, p)$ is $\varepsilon$-allocative efficient on $[0,S]$; 
    \item  
        for all $\tau \in (0, \tau_-]$, 
        the   curve
    $q(s;\ a,\tau, p)$ is $(p - \varepsilon)$-investment efficient on $[0,S].$
\end{enumerate}} 
\begin{proof}
We verify allocative efficiency (Step 1) and investment efficiency (Step 2) in turn.  Let $\varepsilon > 0$, $S> 0$, and $a\geq 0$ be given along with an allocative curve $q$ for a bonding curve $p$. \\

{\bf Step 1.}
By the definition of a limit, it suffices to show there is $\tau^+ \in (0,1)$ such that
\begin{equation}
    q(s, a; \ \tau, p)
    \leq a + \varepsilon,
    \ \ \ \mbox{for all $s \in [0,S]$ and $\tau \in [\tau^+, 1).$}
    \label{eq: proof-efficiency-01}
\end{equation}
By the second property of allocative curves, there is $\overline{\tau}^+ \in (0,1)$ such that
\begin{equation}
    \left| q(S, a;\ \tau, p) - a\right| 
    \leq \varepsilon,
    \ \ \ \mbox{for all $\tau \in [\overline{\tau}^+,1).$}
    \label{eq: proof-efficiency-02}
\end{equation}
With the monotonicity of $q$ with respect to $s$, (\ref{eq: proof-efficiency-02}) implies
\begin{equation}
    q(s, a;\ \tau, p)
    \leq q(S, a;\ \tau,p)
    \leq a + \varepsilon,
    \ \ \ \mbox{for all $\tau \in [\overline{\tau}^+,1).$}
    \label{eq: proof-efficiency-03}
\end{equation}
Thus, (\ref{eq: proof-efficiency-01}) holds, taking $\tau^+=\overline{\tau}^+.$ \\[10pt]

{\bf Step 2.} In similar fashion, here it suffices to show there is $\tau^- \in (0,1)$ such that
\begin{equation}
    q(s, a; \ \tau, p)
    \geq p(s) - \varepsilon,
    \ \ \ \mbox{for all $s \in [0,S]$ and $\tau \in (0,\tau^-].$}
    \label{eq: proof-efficiency-04}
\end{equation}  
Since $q(\cdot,a;\ \tau, p)\rightarrow p$ uniformly as $\tau \rightarrow 0^+$, 
there is $\overline{\tau}^- \in (0,1)$ such that
\begin{equation}
    \|q(\cdot,a;\ \tau,p) - p(s)\|_\infty = \sup_{s\in[0,S]} \left| q(s, a;\ \tau, p) - p(s)\right| 
    \leq \varepsilon,
    \ \ \ \mbox{for all $\tau \in (0,\overline{\tau}^-].$}
    \label{eq: proof-efficiency-05}
\end{equation} 
In particular, this implies (\ref{eq: proof-efficiency-04}) holds, taking $\tau^-=\overline{\tau}^-,$ and the proof is complete.  
\end{proof}

\newpage

\subsection{Allocative Curves for Linear Bonding Curves}
 
A core requirement in the implementation of bonding curves is efficient computation of the payment/reward $m$ and the shares minted/burned $x$.
This is particularly important for smart contract implementations. The lemma below provides two bounds that may be used for estimating $m$ and $x$. 
We emphasize these bounds are given to ensure that transactions incur ``positive friction,'' \ie the errors do not make it possible to game the system by repeating a mint/burn transaction pair profitably. \\

\begin{remark}
    When shares are burned, $x$ and $m$ are \textit{negative}.
    Thus, for burning shares, we seek an underestimate of $-m$ or, equivalently, an overestimate of $m$.
\end{remark}

\begin{lemma}[Allocative Mint and Burn Formulas with Linear Bonding Curves] \label{lemma: allocative-curve-formulas}
    Let $s \geq 0$ be a share supply, $a\geq 0$ be an assessment, and $\tau\in (0,1)$ be a tax rate.
    Let $k > 0$ and $p(s) = k\cdot s$ be a bonding curve and $q$ be the allocative curve for $p$ defined by (\ref{eq: allocative-curve-example}).  
    \begin{enumerate}
        \item In a transaction where shares are burned (\ie $x \in [-s,0)$), the reward $|m|$ is satisfies
        \begin{equation}
            |m| =-m \geq \dfrac{\theta}{1+\tau\cdot k\cdot s}\cdot \max\left(0,\ \  s\cdot |x| - \dfrac{x^2}{\max(2, 1+\tau\cdot k\cdot s)}\right),
        \label{eq: allocative-linear-reward}
        \end{equation}
        and this lower bound on $|m|$ is monotonically increasing as $x$ decreases to $-s$.
        
        \item In a transaction where shares are minted (\ie $x\in (0,\infty)$), the shares minted $x$ satisfy
        \begin{equation}
            x  \geq \dfrac{1+\tau\cdot k\cdot s}{2}\left[ \sqrt{s^2 + \dfrac{4m}{(1-\tau+\tau\cdot a)\cdot k}} - s\right].
            \label{eq: allocative-linear-tokens}
        \end{equation}
    \end{enumerate}
\end{lemma}
\begin{proof}
    We first obtain a formula relating $m$ and $x$ via Taylor's remainder theorem (Step 1).
    This is used to establish (\ref{eq: allocative-linear-reward}) (Step 2) and then (\ref{eq: allocative-linear-tokens}) (Step 3). 

    {\bf Step 1.} First observe 
    \begin{align}
        m  = \int_s^{s+x} q(\zeta, a;\ \tau,p) \ \mbox{d}\zeta  
         = \underbrace{(1-\tau + \tau \cdot a)\cdot k}_{\triangleq \theta} \cdot \int_s^{s+x} \dfrac{  \zeta}{1+\tau\cdot k\cdot \zeta} \ \mbox{d}\zeta.
         \label{eq: lemma-allocative-curve-formula-proof-01}
    \end{align}  
    Integrating yields
    \begin{subequations}
    \begin{align}
       \int_s^{s+x} \dfrac{  \zeta}{1+\tau\cdot k\cdot \zeta} \ \mbox{d}\zeta 
       & = \dfrac{1}{\tau\cdot k}\cdot \int_s^{s+x}  1 - \dfrac{1}{1+\tau\cdot k\cdot \zeta} \ \mbox{d}\zeta \\
       & = \dfrac{1}{\tau\cdot k} \left[ \zeta - \dfrac{\ln(1+\tau\cdot k\cdot \zeta)}{\tau \cdot k}\right]_{s}^{s+x} \\
       & = \dfrac{1}{\tau\cdot k} \left[ x -\dfrac{1}{\tau \cdot k} \cdot \ln\left( \dfrac{1+\tau\cdot k\cdot (s+x)}{1+\tau\cdot k\cdot s}\right) \right] \\
       & = \dfrac{1}{\tau\cdot k} \left[ x -\dfrac{1}{\tau \cdot k}\cdot  \ln\left( 1 + \dfrac{ \tau\cdot k\cdot x}{1+\tau\cdot k\cdot s}\right) \right].
    \end{align}
    \label{eq: lemma-allocative-curve-formula-proof-02}
    \end{subequations}    
    Combining (\ref{eq: lemma-allocative-curve-formula-proof-01}) and (\ref{eq: lemma-allocative-curve-formula-proof-02}) reveals
    \begin{subequations}
    \begin{align}
        m =  \dfrac{\theta}{\tau\cdot k} \left[ x -\dfrac{1}{\tau \cdot k}\cdot  \ln\left( 1 +  \dfrac{ \tau\cdot k\cdot x}{1+\tau\cdot k\cdot s}\right) \right].
    \end{align}
     \label{eq: lemma-allocative-curve-formula-proof-03}
    \end{subequations} 
    Recall the Taylor series expansion relation
    \begin{equation}
        \ln(1+\xi) = \xi - \dfrac{\xi^2}{2} +\dfrac{\xi^3}{3} + \cdots
        = - \sum_{n=1}^\infty \dfrac{(-\xi)^n}{n},
        \ \ \ \mbox{for all $\xi \in (-1,1).$}
         \label{eq: lemma-allocative-curve-formula-proof-04}
    \end{equation} 
    By Taylor's remainder theorem,  there is $\zeta_x $ between $x$ and $0$ such that
    \begin{subequations}
        \begin{align}
        m & = \dfrac{\theta }{\tau\cdot k} \left[ x -  \dfrac{x}{1+\tau\cdot k\cdot s} + \dfrac{ \tau\cdot k\cdot  x^2}{(1+k\cdot\tau\cdot s)^2} -  \dfrac{\tau^2 \cdot k^2 \cdot \zeta_x^3}{(1+k\cdot\tau\cdot s)^3}\right] \\
        & = \theta\left[ \dfrac{  s\cdot x}{1+\tau\cdot k\cdot s} + \dfrac{    x^2}{(1+\tau\cdot k\cdot  s)^2} -  \dfrac{\tau\cdot k \cdot \zeta_x^3}{(1+k\cdot\tau\cdot s)^3}\right]
        \end{align}
         \label{eq: lemma-allocative-curve-formula-proof-05}
    \end{subequations}

    {\bf Step 2.} Using (\ref{eq: lemma-allocative-curve-formula-proof-05}), we see  
    \begin{equation}
        m \leq  \theta    \left[ \dfrac{ s\cdot x}{1+\tau\cdot k\cdot s} + \dfrac{    x^2}{(1+\tau\cdot k\cdot  s)^2}   \right],
        \ \ \ \mbox{for all $x\in[-s,0),$}
         \label{eq: lemma-allocative-curve-formula-proof-06}
    \end{equation}
    noting $-\xi_x^3 > 0$ for $x<0$.
    More simply, the integral in (\ref{eq: lemma-allocative-curve-formula-proof-01}) can also be bounded  by
    \begin{equation}
        m    
        \leq  \dfrac{\theta }{1+\tau\cdot k\cdot  s}\cdot \int_s^{s+x}  \zeta \ \mbox{d}\zeta
        =   \dfrac{\theta \cdot (2\cdot s\cdot x+x^2 )}{2\cdot (1+\tau\cdot k\cdot s)},
        \ \ \ \mbox{for all}\ x\in [-s,0),
        \label{eq: lemma-allocative-curve-formula-proof-07}
    \end{equation}
    where we note $\theta > 0$. 
    Combining (\ref{eq: lemma-allocative-curve-formula-proof-06}) and (\ref{eq: lemma-allocative-curve-formula-proof-07}) reveals, 
    \begin{equation}
        m \leq  \dfrac{\theta}{1+\tau\cdot k\cdot s}\left[ s\cdot x + \dfrac{x^2}{\max(2, 1+\tau\cdot k\cdot s)}\right],
        \ \ \ \mbox{for all $x\in[-s,0).$}
         \label{eq: lemma-allocative-curve-formula-proof-08},
    \end{equation}
    which establishes (\ref{eq: allocative-linear-reward}).
    Lastly, note the vertex of the quadratic bound occurs at
    \begin{equation}
        x = - \dfrac{s\cdot \max(2, 1+\tau\cdot k\cdot s)}{2} \leq -s,
         \label{eq: lemma-allocative-curve-formula-proof-09}
    \end{equation}
    and so the bound (\ref{eq: allocative-linear-reward}) is increasing for all  valid choices of $x$ (\ie $x \geq -s$). \\

    {\bf Step 3.} Setting
    \begin{equation}
        A = \dfrac{1}{(1+\tau\cdot k\cdot s)^2},
        \quad
        B =   \dfrac{s}{1+\tau\cdot k\cdot s},
        \quad
        C = -\dfrac{m}{\theta} - \dfrac{\tau\cdot k\cdot \xi_x^3}{(1+\tau\cdot k\cdot s)^3},
    \end{equation}
    reveals
    \begin{align}
        Ax^2 + Bx + C = 0.
    \end{align}
    Consequently, if $x > 0$, then, by the quadratic formula,
    \begin{subequations}
        \begin{align}
        x & = \dfrac{(1+\tau\cdot k\cdot s)^2}{2}\cdot \left[ \sqrt{ \dfrac{s^2}{(1+\tau\cdot k\cdot s)^2} - 4AC} - \dfrac{s}{1+\tau\cdot k\cdot s}\right] \\
        & = \dfrac{1+\tau\cdot k \cdot s}{2}\left[ \sqrt{s^2 + 4\left( \dfrac{m}{\theta} + \dfrac{\tau\cdot k \cdot \xi_x^3}{(1+\tau\cdot k\cdot s)^3}\right)} - s \right] \\
        & \geq \dfrac{1+\tau\cdot k\cdot s}{2}\left[ \sqrt{s^2 + \dfrac{4m}{\theta}} - s\right],
        \end{align}
    \end{subequations}
    where we note $m > 0$.
    This gives (\ref{eq: allocative-linear-tokens}), as desired.
\end{proof} 
\end{document}